\documentclass[12pt]{article}
 \usepackage{amsmath}
\usepackage{amsfonts}
\usepackage{tikz}
\usepackage[title]{appendix}
\usepackage{subfig}
\usepackage{array}
\usepackage{tabu}
\usepackage[
  separate-uncertainty = true,
  multi-part-units = repeat
]{siunitx}

\usetikzlibrary{arrows}
\usepackage{epsfig,}
\usepackage{epsfig}
\tikzset{
    vertex/.style = {
        circle,
        draw,
        outer sep = 3pt,
        inner sep = 3pt,
    },edge/.style = {->,> = latex'}
}

\usepackage{amssymb}
\usepackage{enumitem}
\usepackage{amsthm}
\usepackage{dsfont}
\def\diag{\mathop{\rm diag}}
\def\adj{\mathop{\rm adj}}
\def\Diag{\mathop{\rm Diag}}

\def\rank{\mathop{\rm rank}}

\def\span{\mathop{\rm span}}

\newcommand{\rr}{\mathbb{R}}
\newcommand{\1}{\mathbf{1}}

\def\a{\alpha}

\def\det{{\rm det}}

\def\tr{{\rm trace}}
\def\csum {{\rm cofsum}}

\def\z{\mathbf{Z}}
\def\zl{\z(\mathcal{L})}

\newtheorem{thm}{Theorem}

\newtheorem{corollary}{Corollary}
\newtheorem{exmm}{Example}

\newtheorem{lem}{Lemma}
\newtheorem{definition}{Definition}

\begin{document}
\begin{center}
\begin{large}
 Resistance matrices of balanced directed graphs
\end{large}
\end{center}
\begin{center}
R. Balaji, R.B. Bapat and Shivani Goel \\
\today
\end{center}

\begin{abstract}
Let $G$ be a strongly connected and balanced directed graph. The Laplacian matrix of $G$ is then the matrix (not necessarily symmetric) $L:=D-A$, where $A$ is the adjacency matrix of $G$ and $D$ is the diagonal matrix such that the row sums and the column sums of $L$ are equal to zero.
Let $L^\dag=[l^{\dag}_{ij}]$ be the Moore-Penrose inverse of $L$.
We define the resistance between any two vertices $i$ and $j$ of $G$ by $r_{ij}:=l^{\dag}_{ii}+l^{\dag}_{jj}-2l^{\dag}_{ij}$.  
In this paper, we derive some interesting properties of the resistance and the corresponding resistance matrix $[r_{ij}]$. 
\end{abstract}

{\bf Keywords.}Balanced directed graph, Laplacian matrix, Moore-Penrose inverse, cofactor sums

{\bf AMS CLASSIFICATION.} 05C50

\section{Introduction}
Let $G=(V,E)$ be a simple connected graph with finite set of vertices $V=\{1,\dotsc,n\}$ and edge set $E$, the set of undirected edges.
To each edge $(i,j) \in E$ assign a weight $w_{ij}$ which is a positive number. 
If $i,j \in V$, define $A:=[a_{ij}]$ where \begin{equation*}
a_{ij} :=\begin{cases}
w_{ij} & (i,j)\in E \\
0 & \text{otherwise}.
\end{cases}
\end{equation*}
Define $D:=\Diag(A \1)$, where $\1$ is the column vector of all ones in $\rr^n$.
The Laplacian matrix of $G$ is then the symmetric matrix $S:=D-A$. If $x \in \rr^n$, then it can be verified that
 \[x^TSx=\sum_{(i,j) \in E} a_{ij} (x_{i}-x_{j})^2 ,\] and hence
$S$ is positive semidefinite with null-space $\span\{\1\}$. The algebraic connectivity of $G$ is the second smallest eigenvalue of the Laplacian matrix $S$ and the associated eigenvector is called the Fiedler vector which  is used to bisect the graph into two connected partitions based on the sign of its components, see Fiedler \cite{fac}.
We shall denote the Moore-Penrose inverse of $S$ by $S^\dag$ and its entries by  $s_{ij}^{\dagger}$.  
To define the distance between any two vertices $i$ and $j$ in $G$, it is natural to consider the length of the shortest path connecting them. This is the classical distance and we shall denote it by $d_{ij}$.  The function $f:V \times V \to \rr $ defined by $f(i,j):=d_{ij}$ is a metric on the vertex set $V$. There are several reasons why the shortest distance $d_{ij}$ is important. In chemistry, the classical distance $d_{ij}$ is used to represent the structure of a molecule as a metric space: see \cite{gutman} and references therein.  
Here is another application in a data communication problem \cite{GP}.
If $v$ is a vertex of $G$, and $N$ is a natural number, define $$A(v):=(a_1,\dotsc,a_N), \mbox{where}~~ a_{i}\in \{0,1,*\}.$$ 
Let $$\rho(A(v),A(v')):=|\{\nu: \{a_{\nu},a_{\nu}^{'}\}=\{0,1\}|~~\forall v,v' \in V.$$ It is known that for some large $N$, there exists 
a function $\rho$ such that 
\begin{equation} \label{ir}
\rho(A(v),A(v'))=d_{vv'}~~\forall v,v' \in V. 
\end{equation}
Now the question is to determine the minimum $N$ for which
equation $(\ref{ir})$ holds. A known result states that 
\[N \geq \max\{n_{+},n_{-}\}, \]
 where $n_{+}$ and $n_{-}$ are the number of positive and negative eigenvalues of the symmetric matrix $[d_{ij}]$: see \cite{Gr}.  
Suppose there are multiple paths connecting $i$ and $j$ in $G$. In a network, this may indicate
that the nodes $i$ and $j$ are better communicated. Thus, it makes more sense to define 
a distance between $i$ and $j$ which is shorter than the classical distance $d_{ij}$.
There are several 
other possible metrics that can be defined on the vertex set $V$ of $G$. In a seminal paper, Klein and R\'andic \cite{kr} introduced the resistance distance $R_{ij}$ between any two vertices $i$ and $j$ of $G$. This is defined  via $S^{\dag}$, the
Moore-Penrose inverse of the Laplacian matrix $S$ of $G$: 
\begin{equation} \label{Rdef}
R_{ij}:=s_{ii}^{\dagger} +s_{jj}^{\dag}-2s_{ij}^{\dag}.
\end{equation}
In resistive electrical networks, $R_{ij}$ is interpreted as the effective electrical resistance between the nodes $i$ and $j$ of a network $N$ corresponding to $G$, with resistor of magnitude $w_{ij}$ taken over the edge $(i,j)$ of $N$.  It can be proved that the resistance distance 
is at most the classical distance and  if $G$ is acyclic, then $R_{ij}=d_{ij}$ for all $i$ and $j$. Resistance distance have several interesting properties. 
These are discussed in chapter $9$ of \cite{bapat}.
In this paper, we generalize the concept of resistance distance to directed graphs.
 
 Let $G=(V,E)$ be a simple directed graph with vertex set $V=\{1,\dotsc,n\}$ and edge set $E$ containing  directed edges. We write $(i,j) \in E$ if there is a directed edge from vertex $i$ to vertex $j$.
  If $i$ and $j$ are any two vertices,  we define
\begin{equation*}
a_{ij} =\begin{cases}
1 & (i,j) \in E \\
0 & \text{otherwise}.
\end{cases}
\end{equation*}
 The matrix $A:=[a_{ij}]$ will be called the adjacency matrix of $G$.
The indegree and the outdegree of a vertex $k$ is 
the sum of all the entries in the $k^{\rm th}$ column and the $k^{\rm th}$ row of the adjacency matrix $A$. 
A vertex $j$ in $V$ is said to be balanced if its indegree and the outdegree are equal. 
Now the graph is said to be balanced if all the vertices are balanced.
Recall that a directed graph is strongly connected, if each pair of vertices is connected by a directed path.
In the sequel, we assume that 
$G$ is a strongly connected and balanced directed graph. The Laplacian of $G$ is now defined by $L:=\Diag(A \1)-A$.
The algebraic connectivity concept is generalized to directed graphs via this definition of the Laplacian matrix and have many other applications like in networks of chaotic systems: see \cite{chai}.
We now propose a {\it semi-distance} in directed graphs using the Moore-Penrose inverse of the Laplacian matrix  $L$. 
\begin{definition} \rm
 The resistance between any two vertices $i$ and $j$ in $V$ is defined by 
\begin{equation} \label{rdist}
r_{ij}:=l_{ii}^{\dag} + l_{jj}^{\dag} -2 l_{ij}^{\dag},
\end{equation}

where $l^\dag_{ij}$ is the $(i,j)^{\rm th}$ entry in the Moore-Penrose inverse of $L$.  
\end{definition}
The matrix $R:=[r_{ij}]$ will be called the resistance matrix
of $G$.  The $\it{reversal}$ of $G$ is the directed graph obtained by reversing the orientation of all the edges. The adjacency matrix of the reversal is then the transpose of the matrix $A$, and thus the resistance matrix of the reversal of $G$ is the transpose of $R$. Because $r_{ij}$ and $r_{ji}$ are not equal in general, $r_{ij}$ is not necessarily a metric 
on $V$ and therefore, the resistance matrices we consider here are not symmetric in general. The symmetric part of the Laplacian 
matrix of $G$ defined by $S:=\frac{1}{2}(L+L')$ has a combinatorial interpretation. Define a simple undirected graph $H$ from $G$ as follows. 
Let the vertex set of $H$ be $V$. If $i,j \in V$, then we shall say that  $i$ and $j$ are adjacent in $H$, if $(i,j) \in E$ or $(j,i) \in E$.
Because $G$ is strongly connected, $H$ is connected. Let $F$ be the set of all edges of $H$. Now to each edge $(i,j) \in F$, define $w_{ij}$ as follows:
\begin{equation*}
w_{ij} =\begin{cases}
1  & (i,j) \in E ~~\mbox{and}~~(j,i) \in E \\
\frac{1}{2} & \text{otherwise}.
\end{cases}
\end{equation*}
Now, $S$ is the Laplacian of the weighted graph $H$. Hence for any $x \in \rr^n$, 
\[ x^TLx= x^TSx=\sum_{(i,j) \in F} w_{ij} (x_i-x_j)^2 .\]
Thus, the null-space of $L$  and null-space of $L'$ are equal to $\span\{\1\}$ and $L+L'$ is positive semidefinite. 
To illustrate, we give an example.

\begin{exmm}\label{12} \rm
Consider the directed graph $G$ with six vertices given in Figure \ref{digrapheg}(a). $G$ is strongly connected and balanced.
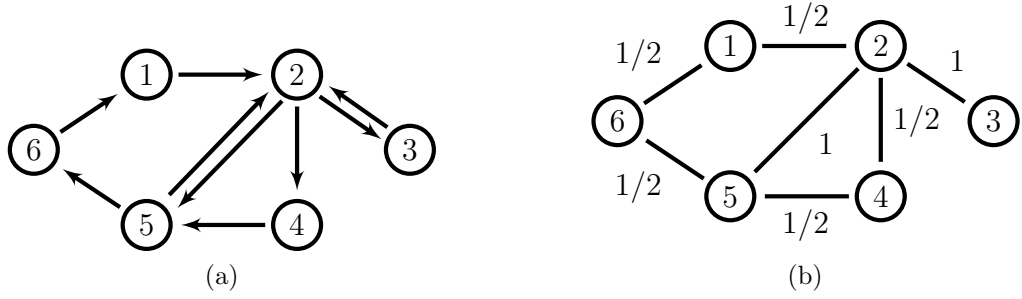
\begin{figure}[!h]
~~~~~~~\subfloat[]{\begin{tikzpicture}[shorten >=1pt, auto, node distance=3cm, ultra thick,
   node_style/.style={circle,draw=black,fill=white !20!,font=\sffamily\Large\bfseries},
   edge_style/.style={draw=black, ultra thick}]
\node[vertex] (1) at  (0,0) {$1$};
\node[vertex] (2) at  (2,0) {$2$};
\node[vertex] (3) at  (3.5,-1) {$3$};
\node[vertex] (4) at  (2,-2) {$4$};
\node[vertex] (5) at  (0,-2) {$5$};
\node[vertex] (6) at  (-1.5,-1) {$6$};
\draw[edge]  (1) to (2);
\draw[edge]  (2.3,-0.29) to (3.15,-0.88);
\draw[edge]  (3.2,-0.7) to (2.39,-0.13);
\draw[edge]  (0.3,-1.6) to (1.65,-0.2);
\draw[edge]  (2) to (4);
\draw[edge]  (1.8,-0.35) to (0.385,-1.8);
\draw[edge]  (4) to (5);
\draw[edge]  (5) to (6);
\draw[edge]  (6) to (1);
\end{tikzpicture}}
~~~~~~~~~~~~~~ 
\subfloat[]{\begin{tikzpicture}[shorten >=1pt, auto, node distance=3cm, ultra thick,
   node_style/.style={circle,draw=black,fill=white !20!,font=\sffamily\Large\bfseries},
   edge_style/.style={draw=black, ultra thick}]
\node[vertex] (1) at  (0,0) {$1$};
\node[vertex] (2) at  (2,0) {$2$};
\node[vertex] (3) at  (3.5,-1) {$3$};
\node[vertex] (4) at  (2,-2) {$4$};
\node[vertex] (5) at  (0,-2) {$5$};
\node[vertex] (6) at  (-1.5,-1) {$6$};
\draw  (1) edge node {1/2} (2);
\draw (2) edge node {1} (3);
\draw  (2) edge node {1} (5);
\draw  (2) edge node {1/2} (4);
\draw  (4) edge node {1/2} (5);
\draw  (5) edge node {1/2} (6);
\draw  (6) edge node {1/2} (1);
\end{tikzpicture}}
\caption{(a) Graph $G$ and (b) Graph $H$} \label{digrapheg}
\end{figure} 
The adjacency and the Laplacian matrices of $G$ are:
\begin{small}
\begin{equation*}\label{adj&lap}
A = 
\left[
{\begin{array}{rrrrrr}
0 & 1 & 0 & 0 & 0 & 0 \\
0 & 0 & 1 & 1 & 1 & 0 \\
0 & 1 & 0 & 0 & 0 & 0 \\
0 & 0 & 0 & 0 & 1 & 0 \\
0 & 1 & 0 & 0 & 0 & 1 \\
1 & 0 & 0 & 0 & 0 & 0
\end{array}}
\right]~~\mbox{and}~~ 
L = 
\left[
{\begin{array}{rrrrrr}
1 & -1 & 0 & 0 & 0 & 0 \\
0 & 3 & -1 & -1 & -1 & 0 \\
0 & -1 & 1 & 0 & 0 & 0 \\
0 & 0 & 0 & 1 & -1 & 0 \\
0 & -1 & 0 & 0 & 2 & -1 \\
-1 & 0 & 0 & 0 & 0 & 1
\end{array}}
\right].
\end{equation*} 
\end{small}The Moore-Penrose inverse of $L$ is
\begin{small}
\begin{equation*}\label{pinv}
L^\dag = 
\left[
{\begin{array}{ccccccccc}
\frac{5}{9} & \frac{1}{18} & -\frac{1}{9} & -\frac{1}{9} & -\frac{1}{9} & -\frac{5}{18} \\
-\frac{5}{18} & \frac{2}{9} & \frac{1}{18} & \frac{1}{18} & \frac{1}{18} & -\frac{1}{9} \\
-\frac{4}{9} & \frac{1}{18} & \frac{8}{9} & -\frac{1}{9} & -\frac{1}{9} & -\frac{5}{18} \\
-\frac{7}{36} & -\frac{7}{36} & -\frac{13}{36} & \frac{23}{36} & \frac{5}{36} & -\frac{1}{36} \\
-\frac{1}{36} & -\frac{1}{36} & -\frac{7}{36} & -\frac{7}{36} & \frac{11}{36} & \frac{5}{36} \\
\frac{7}{18} & -\frac{1}{9} & -\frac{5}{18} & -\frac{5}{18} & -\frac{5}{18} & \frac{5}{9}
\end{array}}
\right].
\end{equation*}
\end{small}The resistance matrix $R=[r_{ij}]=[l^{\dag}_{ii} +l^\dag_{jj}-2l^\dag_{ij}]$ is given by
\begin{small}
\begin{equation*}\label{res}
R = 
\left[
{\begin{array}{rrrrrr}
0 & \frac{2}{3} & \frac{5}{3} & \frac{17}{12} & \frac{13}{12} & \frac{5}{3} \\
\frac{4}{3} & 0 & 1 & \frac{3}{4} & \frac{5}{12} & 1 \\
\frac{7}{3} & 1 & 0 & \frac{7}{4} & \frac{17}{12} & 2 \\
\frac{19}{12} & \frac{5}{4} & \frac{9}{4} & 0 & \frac{2}{3} & \frac{5}{4} \\
\frac{11}{12} & \frac{7}{12} & \frac{19}{12} & \frac{4}{3} & 0 & \frac{7}{12} \\
\frac{1}{3} & 1 & 2 & \frac{7}{4} & \frac{17}{12} & 0
\end{array}}
\right].
\end{equation*}
\end{small}The undirected graph $H$ obtained from $G$  is given in Figure \ref{digrapheg}(b). The Laplacian matrix $S$ of $H$ is given by
\begin{small}
\begin{equation*}
S = 
\left[
{\begin{array}{rrrrrr}
1 & -\frac{1}{2} & 0 & 0 & 0 & -\frac{1}{2} \\
-\frac{1}{2} & 3 & -1 & -\frac{1}{2} & -1 & 0 \\
0 & -1 & 1 & 0 & 0 & 0 \\
0 & -\frac{1}{2}& 0 & 1 & -\frac{1}{2} & 0 \\
0 & -1 & 0 & -\frac{1}{2} & 2 & -\frac{1}{2} \\
-\frac{1}{2} & 0 & 0 & 0 & -\frac{1}{2}& 1
\end{array}}
\right].
\end{equation*}
\end{small}
It can be verified that $S = \frac{1}{2} (L+L')$. 
\end{exmm}

Suppose $G'=(V,F)$ is a simple undirected and connected graph. Let $[R_{ij}]$
be the resistance matrix of $G'$, where $R_{ij}$ is defined in $(\ref{Rdef})$.
Now $[R_{ij}]$ is the resistance matrix of a strongly connected and balanced directed graph.
To see this, we proceed as follows.
Let $L$ be the Laplacian matrix of $G'$.
 From the edge set $F$, we shall define
 a set of directed edges.  For each edge $(i,j) \in F$, define two directed edges, viz,
$(i,j)$ and $(j,i)$ and let 
$E'$ be the set of all such directed edges.
Then the directed graph $G':=(V,E')$ is strongly connected and balanced. It can be easily seen that the adjacency matrices of $G$ and $G'$ are equal and hence their Laplacian matrices are equal. This means that between any two vertices $i$ and $j$, the resistance distance in $G$ and the resistance in $G'$ defined by $(\ref{Rdef})$ and $(\ref{rdist})$, respectively are same. To illustrate, we give an example.

\begin{exmm}\label{conngraph}\rm
Let $G$ be the graph with five vertices given in Figure \ref{conngrapheg}(a).
\begin{figure}[!h]
~~~~~~~\subfloat[]{\begin{tikzpicture}[shorten >=1pt, auto, node distance=3cm, ultra thick,
   node_style/.style={circle,draw=black,fill=white !20!,font=\sffamily\Large\bfseries},
   edge_style/.style={draw=black, ultra thick}]
\node[vertex] (1) at  (1,0) {$1$};
\node[vertex] (2) at  (3.5,-1) {$2$};
\node[vertex] (3) at  (2,-2.5) {$3$};
\node[vertex] (4) at  (0,-2.5) {$4$};
\node[vertex] (5) at  (-1.5,-1) {$5$};
\draw  (1) to (2);
\draw  (1) to (3);
\draw  (1) to (4);
\draw  (2) to (3);
\draw  (3) to (4);
\draw  (4) to (5);
\draw  (5) to (1);
\end{tikzpicture}}
~~~~~~~~~~~~~~
\subfloat[]{\begin{tikzpicture}[shorten >=1pt, auto, node distance=3cm, ultra thick,
   node_style/.style={circle,draw=black,fill=white !20!,font=\sffamily\Large\bfseries},
   edge_style/.style={draw=black, ultra thick}]
\node[vertex] (1) at  (1,0) {$1$};
\node[vertex] (2) at  (3.5,-1) {$2$};
\node[vertex] (3) at  (2,-2.5) {$3$};
\node[vertex] (4) at  (0,-2.5) {$4$};
\node[vertex] (5) at  (-1.5,-1) {$5$};
\draw[edge]  (1) to (2);
\draw[edge]  (3.2,-0.72) to (1.4,0);
\draw[edge]  (1) to (3);
\draw[edge]  (2,-2) to (1.3,-0.3);
\draw[edge]  (1) to (4);
\draw[edge]  (0,-2) to (0.7,-0.3);
\draw[edge]  (5) to (1);
\draw[edge]  (0.6,0) to (-1.25,-0.72);
\draw[edge]  (2) to (3);
\draw[edge]  (2.4,-2.36) to (3.34,-1.41);
\draw[edge]  (3) to (4);
\draw[edge]  (0.4,-2.65) to (1.6,-2.65);
\draw[edge]  (4) to (5);
\draw[edge]  (-1.3,-1.4) to (-0.4,-2.35);
\end{tikzpicture}}
\caption{(a) Graph $G$ and (b) Graph $G'$} \label{conngrapheg}
\end{figure}
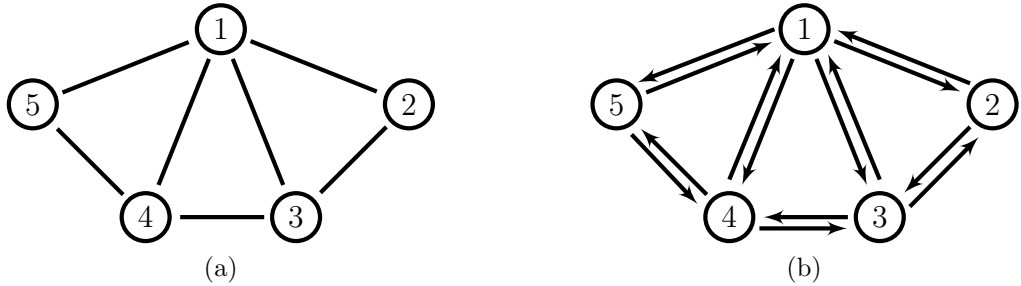 
The directed graph $G'$ constructed from $G$ is shown in Figure \ref{conngrapheg}(b). The adjacency and the Laplacian matrices of $G$ and $G'$ are
given by 
\begin{small}
\begin{equation*}
A = 
\left[
{\begin{array}{rrrrrr}
0 & 1 & 1 & 1 & 1 \\
1 & 0 & 1 & 0 & 0 \\
1 & 1 & 0 & 1 & 0 \\
1 & 0 & 1 & 0 & 1 \\
1 & 0 & 0 & 1 & 0
\end{array}}
\right]~~
\text{and}~~
L = 
\left[
{\begin{array}{rrrrrr}
4 & -1 & -1 & -1 & -1 \\
-1 & 2 & -1 & 0 & 0 \\
-1 & -1 & 3 & -1 & 0 \\
-1 & 0 & -1 & 3 & -1 \\
-1 & 0 & 0 & -1 & 2
\end{array}}
\right].
\end{equation*}
\end{small}
\end{exmm}

\subsection{Results obtained in the paper}
\begin{itemize}
\item In our first result, we show that the resistance $r_{ij}$ defined in $(\ref{rdist})$ has the following properties.
\begin{enumerate}
 \item[{\rm (i)}] If $i$ and $j$ are any two distinct vertices of $G$, then $r_{ij}> 0$. 
  \item[{\rm (ii)}] If $i,j,k$ are any three vertices, then
\[r_{ik} \leq r_{ij}+r_{jk}~~~\forall i,j,k.  \]
\end{enumerate}

\item In our next result, we compute an identity for the
inverse of the resistance matrix $[r_{ij}]$. The motivation for obtaining this identity starts from a classical result of Graham and Lov\'{a}sz \cite{Gr}. This states the following.
\begin{thm} \label{GL}
Let $T$ be a tree with $V(T)=\{1,\dotsc,n\}$. Let $d_{ij}$ be 
the length of the shortest path between vertices $i$ and $j$, 
and $L$ be the Laplacian of $T$. Set $D:=[d_{ij}]$. Then,
\[D^{-1}=-\frac{1}{2} L + \frac{1}{2(n-1)}\tau \tau', \]
where $\tau=(2-\delta_1,\dotsc,2-\delta_n)'$ and $\delta_i$ is the degree of the vertex $i$.
\end{thm}

Theorem $\ref{GL}$ is extended to connected graphs in \cite{rbbres} for resistance matrices. 
\begin{thm} \label{invres}
Let $G$ be a simple connected graph with vertex set $V=\{1,\dotsc,n\}$ and edge set $E$. Let $S$ be the Laplacian of $G$ and $R_{ij}$ be the resistance distance defined in $(\ref{Rdef})$. Define $\widetilde{R}:=[R_{ij}]$.
Then,
\[\widetilde{R}^{-1}=-\frac{1}{2}S+ \frac{1}{\tau'\widetilde{R}\tau} \tau \tau', \]
where $\tau_{i}=2- \sum_{(i,j) \in E} R_{ij}$. 
\end{thm}

Motivated by the above two results, we find the following inverse formula for the resistance matrix $[r_{ij}]$. 
\begin{thm} \label{invr}
Let $G=(V,E)$ be a strongly connected and balanced directed graph. 
Let $r_{ij}$ be the resistance between the vertices $i$ and $j$ defined in $(\ref{rdist})$ and $R:=[r_{ij}]$. Then,
\begin{equation*}
R^{-1} = - \frac{1}{2} L+ \frac{1}{\tau'R \tau} (\tau ({\tau}'+1'\diag(L^\dag)M)),
\end{equation*}
where $M=L-L'$, and $\tau_{i}:=2-\sum_{(i,j) \in E} r_{ji}$.
\end{thm}
Since the resistance matrix of a connected graph can be written as a resistance matrix of a strongly connected and balanced directed graph,  Theorem \ref{GL} and \ref{invres} are special cases of Theorem \ref{invr}. Using Theorem \ref{invr}, we find a formula for computing $\det(R)$.
\item In our final result, we investigate the sum of all the cofactors
in an $s \times s$ submatrix of $R=[r_{ij}]$. The motivation for this consideration comes from
an alternate method to compute the resistance distance defined in (\ref{Rdef}).  This method gives an elegant formula to compute $R_{ij}$:
\begin{equation} \label{Rij}
R_{ij}=\frac{1}{\delta} {\det(S (\{i,j\},\{i,j\}))}, 
\end{equation}
where $S(\{i,j\},\{i,j\})$ is the principal submatrix of $S$ obtained by deleting rows and columns indexed by $\{i,j\}$ and $\delta$ is the number of spanning trees in $G$.
A far reaching generalization of
(\ref{Rij}) is obtained in \cite{bs}. This is stated below.
\begin{thm} \label{csumR}
Let $G$ be a connected graph with vertex set $\{1,2,\ldots,n\}$. Let $S$ be the Laplacian matrix of $G$ and $\widetilde{R}:=[R_{ij}]$ its resistance matrix. Let $\Omega_1,\Omega_2 \subset \{1,2,\ldots,n\}$ be non-empty, and let $ |\Omega_1| = |\Omega_2|$. Put $\eta:=|\Omega_1|$. Suppose $\alpha(\Omega_1)$ and $\alpha(\Omega_2)$ are the sum of all the elements in $\Omega_1$ and $\Omega_2$, respectively.
Let $S[\Omega_1,\Omega_2]$ denote the $\eta \times \eta$ submatrix of $S$ with rows and columns 
indexed by $\Omega_1$ and $\Omega_2$, respectively, and $\widetilde{R}[\Omega_2^{c},\Omega_1^{c}]$ be the 
$(n-\eta) \times (n-\eta)$ submatrix of $R$ with rows and columns indexed by $\Omega_2^c$ and 
$\Omega_1^c$, respectively. 
Then,
\begin{equation} \label{csumRtilde}
 \csum(\widetilde{R}[\Omega_1,\Omega_2]) = (-1)^{\alpha(\Omega_1)+\alpha(\Omega_2)+\eta-1}  \frac{2^{\eta-1}}{\delta} \det(S[\Omega_2^c,\Omega_1^c]),
 \end{equation}
where $\delta$ is the number of spanning trees in $G$.
\end{thm}
 Equation $(\ref{Rij})$ is a special case of   $(\ref{csumRtilde})$.  This follows by setting 
 $\Omega_1=\Omega_2=\{i,j\}$ and observing that 
$(-1)^{\a(\Omega_1)+\a(\Omega_2)}=1$, $\csum(\widetilde{R} [\Omega_1,\Omega_2])=-2R_{ij}$,
$\eta=2$ and $S[\Omega_2^c,\Omega_1^c]=S(\{i,j\},\{i,j\})$.
In this paper, we generalize Theorem \ref{csumR} to resistance matrices of directed graphs. 
\end{itemize}

\subsection{Outline of the paper}
In section $2$, we mention the preliminaries that are needed for further discussion. 
In section $3$, we discuss the properties of the resistance.
In section $4$, we present the inverse formula stated in Theorem \ref{invr} and illustrate it
by an example. In the final section, we deduce a formula for finding the cofactor sums of the resistance matrix.

\section{Preliminaries}
We now list a few notation used in this paper and gather some tools to prove our results.
\begin{enumerate}
\item[(P1)] Let $\Omega_1$ and $\Omega_2$ be non-empty subsets of $\{1,\dotsc,n\}$. 
If $W$ is an $n \times n$
matrix, then $W[\Omega_1,\Omega_2]$ will be the submatrix of $W$ with rows and columns indexed by $\Omega_1$ and $\Omega_2$, respectively. If $\Omega \subseteq \{1,\dotsc,n\}$ is non-empty, then
$\a(\Omega)$ will denote the sum of all elements in $\Omega$. 

\item[(P2)] The complement of a set $\Omega$ is written $\Omega^c$. The transpose and the Moore-Penrose inverse of a matrix $A$ are denoted by 
$A'$ and $A^{\dag}$, respectively. All vectors are regarded as column vectors.

\item[(P3)] If $A=[a_{ij}]$ is a square matrix, then $\diag(A)$ is the diagonal matrix with diagonal entries equal to $a_{ii}$. If 
$s:=(s_1,s_2,\dotsc,s_n)' \in \rr^n$, then $\Diag(s)$ will be the diagonal matrix with diagonal entries equal to $s_{i}$. 

\item[(P4)] The sum of all the cofactors of an $m \times m$ matrix $A$ is represented by $\csum(A)$. The determinant and the classical adjoint of $A$ are written $\det(A)$ and $\adj(A)$, respectively. 

\item[(P5)] The notation $\1$ will stand for the vector $(1,1,\dotsc,1)'$ in $\rr^n$ and $J:=\1 \1'$. The orthogonal projection onto the hyperplane 
$\{\1\}^{\perp}$ is denoted by $P$. It is easy to observe that $P=I-\frac{1}{n}J$, where $I$ is the $n \times n$ identity matrix.
If $1 \leq m<n$, then the vector of all ones in $\rr^m$ and the $m \times m$ identity matrix will be denoted by $\1_m$ and $I_m$, respectively.

\item[(P6)] The Jacobi's identity on non-singular matrices is the following: \begin{thm}
Let $A$ be an $n \times n$ non-singular matrix. 
Let $\Omega_1$, $\Omega_2 \subset \{1,\dotsc,n\}$ be non-empty such that $|\Omega_1|=|\Omega_2|$.
Then,
\[\det(A^{-1}[\Omega_2^{c},\Omega_1^c])= (-1)^{\alpha(\Omega_2) + \a (\Omega_1)}\frac{1}{\det(A)} {\det(A[\Omega_1,\Omega_2]}). \] 
\end{thm} See Brualdi and Schneider \cite{bru}.

\item[(P7)] An $n \times n$ matrix $B$ is called a $\z$-matrix, if every off-diagonal entry of $B$  is non-positive. If $L$ is the Laplacian matrix of a strongly connected and balanced directed graph, then $L$ is a $\z$-matrix. As already noted, $L+L'$ is positive semidefinite, $L\1 = L' \1 = 0$ and $\rank (L) = n-1$. 
\item[(P8)] Suppose $S$ is an $n \times n$ matrix such that $S\1 = S'\1 =0$ and $\rank(S) = n-1$. Then $S^\dag \1 = {S^\dag}'\1 = 0$, $SS^{\dag}=S^{\dag}S=P=I-\frac{1}{n}J$ and all the cofactors of $S$ are equal. If $L$ is a $\z$-matrix such that $L\1 = L' \1 = 0$ and $\rank (L) = n-1$, then we shall write $L \in \zl$. If $L \in \zl$, then it can be verified that $L^\dag+{L^\dag}'$ is positive semidefinite and $\tr(L^\dag) > 0$.

\item[(P9)] Let $A$ be an $n \times n$ matrix. If $u$ and $v$ belong to $\rr^n$, then $\det(A+uv') = \det(A)+v'\adj(A)u$. For a proof, see Lemma $1.1$ in \cite{jiu}.

\item[(P10)] Let $B=[b_{ij}]$ be an $n \times n$ matrix. Then, 
\begin{enumerate}
\item $B$ is row diagonally dominant if for each $i=1,\dotsc,n$ 
\[|b_{ii}| \geq \sum_{\{j:i \neq j\}} |b_{ij}|~~\forall j=1,\dotsc n.\] 
\item $B$ is diagonally dominant of its row entries if 
\[|b_{ii}| \geq |b_{ij}| \]
for each $i=1,\dotsc,n$ and $j \neq i$.
\item $B$ is diagonally dominant of its column entries if $B'$ is diagonally dominant of its row entries. 
\end{enumerate}
By Theorem $2.5.12$ in \cite{horn}, if $B$ is non-singular and row diagonally dominant, then $B^{-1}$ is diagonally dominant of its column entries.
\item [(P11)]  Let $G = (V,E)$ be a directed graph with vertex set $V = \{1,2,...,n\} $. An \emph{oriented spanning tree} of $G$ rooted at vertex $i$ is a spanning subgraph $T$ such that
\begin{enumerate}
\item [(i)] Every vertex $j$ of $T$ such that $j \neq i$ has outdegree $1$.
\item [(ii)] The vertex $i$ has outdegree $0$.
\item [(iii)] $T$ has no oriented cycles. 
\end{enumerate} The matrix-tree theorem for directed graphs (Theorem $1$ in \cite{pat}) is the following. 
\begin{thm}
 Let $G = (V,E)$ be a directed graph with vertex set $V = \{1,2,...,n\} $.  Let  $\kappa(G,i)$ denote the number of oriented spanning trees of $G$ rooted at $i$. If $L$ is the Laplacian matrix of $G$, then 
\[ \kappa(G,i) = \det(L[\{i\}^c,\{i\}^c]).\]
\end{thm}
Suppose $G$ is also strongly connected and balanced. Then all the cofactors of $L$ are equal  and therefore $\kappa(G,i)$ is independent of $i$. We denote $\kappa(G,i)$ by $\kappa(G)$ in the rest of the paper.
\end{enumerate}
\section{Properties of the resistance }
To establish the desired properties of the resistance defined in $(\ref{rdist})$, we need the following identity. The proof is omitted as it is a direct verification.
\begin{lem}\label{3.3}
Let $L \in \zl$. Then $L$ can be partitioned as
\[L =\left[
\begin{array}{cccc}
B & -B e \\
-e' B  & e'B e \\
\end{array}
\right],\]
where $B$ is a square matrix of order $n-1$ and $e = \1_{n-1}$ and 
\[L^{\dag} = \left[
\begin{array}{cccc}
B^{-1} - \frac{1}{\displaystyle n} e e' B^{-1} - \frac{1}{\displaystyle n} B^{-1} ee' & - \frac{1}{\displaystyle n} B^{-1} e \\ \\
-\frac{1}{\displaystyle n} e' B^{-1} & 0 \\ 
\end{array}
\right] + \frac{e' B^{-1}e}{\displaystyle n^2} \1\1'.\]
\end{lem}

The following theorem is an application of Lemma \ref{3.3}.

\begin{thm}\label{15}
Let $L \in \zl$, $L:=[l_{ij}]$ and $L^\dag:=[l_{ij}^{\dag}]$. Define $r_{ij}:=l^{\dag}_{ii} + l^\dag_{jj}- 2l^\dag_{ij}$. Then,
\begin{enumerate}
\item[\rm(i)] $r_{ij}>0~~\forall i \neq j$.
\item[\rm(ii)] $r_{ik} \leq r_{ij}+r_{jk}~~~\forall i,j,k.$
\end{enumerate}
\end{thm}
\begin{proof}
Define $\Omega:=\{1,\dotsc,n-1\}$, $B:=L[\Omega,\Omega]$ and $C:= B^{-1}=[c_{ij}]$. To prove (i), we shall assume without loss of generality that $j=n$ and show that $r_{in} >0$ for any $i \in \Omega$. Put $e:=\1_{n-1}.$
By Lemma $\ref{3.3}$,
\begin{equation}\label{eqn33}
L^{\dag} = \left[
\begin{array}{cccc}
B^{-1} - \frac{1}{\displaystyle n} e e' B^{-1} - \frac{1}{\displaystyle n} B^{-1} ee' & - \frac{1}{\displaystyle n} B^{-1} e \\ \\
-\frac{1}{\displaystyle n} e' B^{-1} & 0 \\ 
\end{array}
\right] + \frac{e' B^{-1}e}{\displaystyle n^2} \1\1'.
\end{equation} 
By a well-known result on $\z$-matrices, $B^{-1}$ is a non-negative matrix. Therefore, $B^{-1} e$ is a positive vector.   
Let $x:= B^{-1} e$ and $y':= e'B^{-1}$. For any $i \in \Omega$,  by $(\ref{eqn33})$ we have
\begin{equation} \label{rin}
\begin{aligned}
r_{in} &= l^{\dag}_{ii} + l^{\dag}_{nn} - 2l^{\dag}_{in} 
\\ &= c_{ii} - \frac{1}{n} y_{i} - \frac{1}{n} x_{i} + \frac{2}{n} x_i 
\\ &= c_{ii} - \frac{1}{n} y_i + \frac{1}{n} x_i.
\end{aligned}
\end{equation} 
It can be seen that $B$ is row diagonally dominant. In view of (P10), $C$ is diagonally dominant of its column entries and therefore, $$c_{ii} \geq c_{ji}~~ \forall j = 1,\dotsc , n-1.$$ Thus, $$nc_{ii} \geq (n-1)c_{ii} \geq \sum_{j=1}^{n-1} c_{ji}=y_i.$$ Hence, $$c_{ii} \geq \frac{y_i}{n}.$$ Since $x_{i} > 0$, it follows from $(\ref{rin})$ that $r_{in} > 0$. This completes the proof of (i). \\
We now prove (ii). We shall show that if $j,k \in \Omega$, then
\[r_{nk} \leq r_{nj}+r_{jk},\]
and the proof can be completed by using a similar argument applied to any other $r_{ik}$. Since
\begin{equation*}\label{eqn34}
\begin{aligned}
r_{nk} - r_{nj}-r_{jk} &= l_{nn}^\dag+l_{kk}^\dag-2l_{nk}^\dag-l_{nn}^\dag-l_{jj}^\dag+2l_{nj}^\dag-l_{jj}^\dag-l_{kk}^\dag+2l_{jk}^\dag \\ &= -2(l_{nk}^\dag+l_{jj}^\dag -l_{nj}^\dag-l_{jk}^\dag),
\end{aligned}
\end{equation*}
it suffices to show that $l_{nk}^\dag+l_{jj}^\dag -l_{nj}^\dag-l_{jk}^\dag \geq 0$. In view of $(\ref{eqn33})$, it follows that
\begin{equation}\label{eqn35}
\begin{aligned}
l_{nk}^\dag+l_{jj}^\dag -l_{nj}^\dag-l_{jk}^\dag &= -\frac{1}{n}y_k+c_{jj}-\frac{1}{n}y_j-\frac{1}{n}x_j+\frac{1}{n}y_j-c_{jk}+\frac{1}{n}y_k+\frac{1}{n}x_j \\ &= c_{jj}-c_{jk}.
\end{aligned}
\end{equation}
Since $B'$ is row diagonally dominant, by (P10), $C$ is diagonally dominant of its row entries, and hence $c_{jj} \geq c_{jk}$. The proof is complete. 
\end{proof}
The main result of this section is now immediate from the above result. 
\begin{thm} \label{nr}
Let $G$ be a strongly connected and balanced directed graph and $R:=[r_{ij}]$ be the resistance matrix of $G$. Then, every off-diagonal entry of
$R$ is positive and thus $R$ is a non-negative matrix. Furthermore, the resistance $r_{ij}$ satisfies the triangle inequality.
\end{thm}
\section{Inverse of the resistance matrix}
For a resistance matrix  $R$, we now obtain the inverse formula stated in Theorem \ref{invr}. Since $r_{ij}=l_{ii}^{\dag} + l_{jj}^{\dag} -2 l_{ij}^{\dag}$ and $R = [r_{ij}]$, we have
\begin{equation} \label{req}
R=\diag(L^{\dag}) J + J \diag(L^{\dag}) -2 L^{\dag}. 
\end{equation} Define $X: = ( L + \frac{1}{n}J ) ^{-1}$ and $\tilde{X}:=\diag(X)$. By an easy computation, we find that 
$L^{\dag} = X - \frac{1}{n}J$ and hence
\[ R= \tilde{X}J+J\tilde{X}-2X.  \]
For $i = 1,2,..,n$, let
\begin{equation*}\label{taudef}
{\tau}_i := 2-\sum_{\{j:(i,j) \in E\}}{r_{ji}}~~\mbox{and}~~\tau:=(\tau_1,\dotsc,\tau_n)'.
\end{equation*}
Set $M:= L - L'$.
The inverse formula will be proved by using the following lemma.
\begin{lem}\label{3}
The following are true.
 \begin{enumerate}
\item[{\rm{(i)}}] $\tau = L\tilde{X}\1+\frac{2}{n}\1$.
\item[{\rm (ii)}] $\tau'+\1'\tilde{X}M = \1'\tilde{X}L+\frac{2}{n}\1'$.
\item[{\rm (iii)}] $LR+2I = \tau\1'$.
\item[{\rm (iv)}] $RL +2I = \1\tau'+J\tilde{X}M.$
\item[{\rm (v)}] $\1' \tau=2$.
\item[{\rm (vi)}] \(\tau'R\tau = 2\tilde{x}^{\prime}L\tilde{x}+\frac{8}{n}\tr(L^{\dag}).\)
\item[{\rm (vii)}] $\tau'R \tau >0$. 
\end{enumerate}
\end{lem}

\begin{proof}
Fix $i \in \{1,\dotsc,n\}$. 
Define 
$\delta_i:=(A\1)_{i}$.
From \mbox{$\big(L + \frac{1}{n} J)X = I$}, we have
\begin{equation} \label{one}
\delta_{i} x_{ii} - \sum_{\{j:(i,j) \in E\}}{x_{ji}} + \frac{1}{n} \sum_{j=1}^{n}{x_{ji}} = 1.
\end{equation}
As $X\1=X' \1=\1$, 
\[ \sum_{j=1}^{n}{x_{ji}} = (X' \1)_{i}=1.\]
Hence from $(\ref{one})$, 
\begin{equation*}
\delta_i x_{ii} - \sum_{\{j:(i,j) \in E\}}{x_{ji}} = 1 - \frac{1}{n}.
\end{equation*}
and so,
\begin{equation} \label{two}
\sum_{\{j:(i,j) \in E\}}{x_{ji}} = \delta_i x_{ii} - 1 + \frac{1}{n}.
\end{equation}
Also, we see that
\begin{equation}\label{three}
\begin{aligned}
{\tau}_i &= 2- \sum_{\{j:(i,j) \in E\}}{r_{ji}} \\
&=2- \sum_{\{j:(i,j) \in E\}}{(x_{ii}+x_{jj}-2x_{ji})} \\
&=2- \sum_{\{j:(i,j) \in E\}}{x_{ii}} - \sum_{\{j:(i,j) \in E\}}{x_{jj}+2 \sum_{\{j:(i,j) \in E\}}{x_{ji}}}.
\end{aligned}
\end{equation}
Since
\begin{equation} \label{cc}
\sum_{\{j:(i,j) \in E\}} x_{ii}=x_{ii} \sum_{j=1}^na_{ij}, 
\end{equation}
 and 
 \begin{equation} \label{ccc}
\sum_{j=1}^na_{ij}=(A \1)_i=\delta_i,
\end{equation}
from $(\ref{three})$,  $(\ref{cc})$ and $(\ref{ccc})$, we now have
\[\tau_i=2- \delta_{i}{x_{ii}} - \sum_{\{j:(i,j) \in E\}}{x_{jj}}+2 \sum_{\{j:(i,j) \in E\}}{x_{ji}} .\]
In view of $(\ref{two})$, 
\begin{equation} \label{tau}
\begin{aligned} 
 \tau_i &=2- \delta_{i}{x_{ii}} - \sum_{\{j:(i,j) \in E\}}{x_{jj}}+ 2\delta_{i}x_{ii} -2 + \frac{2}{n}  \\
 &=\delta_i{x_{ii}} -\sum_{\{j:(i,j) \in E\}}{x_{jj}}+ \frac{2}{n} .\\
 \end{aligned}
 \end{equation}
Since
\begin{equation}\label{eqn8}
\begin{aligned}
\sum_{\{j:(i,j) \in E\}} x_{jj}&=\sum_{j=1}^n a_{ij} x_{jj} \\
&=(A\tilde{X} \1)_i \\
\end{aligned}
\end{equation}and 
\begin{equation}\label{eqn9}
(\Diag(A\1) \tilde{X} \1)_{i}=\delta_{i}x_{ii},
\end{equation} equations  $(\ref{tau})$, $(\ref{eqn8})$ and $(\ref{eqn9})$ imply
\begin{equation*} \label{taufinal}
\begin{aligned}
\tau_i &= ((\Diag(A\1) -A)\tilde{X} \1)_i + \frac{2}{n} \\
&=(L \tilde{X} \1 + \frac{2}{n} \1)_{i}.
\end{aligned}
\end{equation*}
Thus,
\[\tau=L \tilde{X} \1 + \frac{2}{n} \1. \]
The proof of (i) is complete. \\
We have 
\begin{equation*}
\begin{aligned}
\1' \tilde{X} M + \tau'&=\1' \tilde{X}L-\1'\tilde{X}L'+\1^{\prime} \tilde{X} L'+\frac{2}{n} \1' \\
&=\1' \tilde{X}L +\frac{2}{n} \1'.
\end{aligned}
\end{equation*}
The proof of (ii) is complete. \\
To prove (iii), recall that
$$R=\tilde{X} J + J\tilde{X} -2X.$$ As $X=L^{\dag}+\frac{1}{n}J$, we have
\begin{equation} \label{R}
R=\tilde{X}J+J\tilde{X}-2L^{\dag}-\frac{2}{n}J
\end{equation}
In view of (P8), $LL^{\dag}=I-\frac{1}{n}J$. Since $LJ=0$, by $(\ref{R})$, 
\begin{equation}\label{eqn10}
\begin{aligned}
LR&= L \tilde{X} J -2I +\frac{2}{n} J \\
&=L \tilde{X} \1 \1' -2 I + \frac{2}{n} \1 \1' \\
&=(L \tilde{X} \1+\frac{2}{n} \1)\1' -2 I.
\end{aligned}
\end{equation}
By (i), $$\tau=L \tilde{X} \1 + \frac{2}{n} \1.$$
Hence by $(\ref{eqn10})$, $LR= \tau \1'-2I$. This completes the proof of (iii). \\
To prove (iv), first we observe that
\begin{equation} \label{RL}
\begin{aligned}
RL + 2I&= (\tilde{X} J + J \tilde{X} -2X) L+ 2I \\
&=J \tilde{X}L - 2XL+2I .\\ 
\end{aligned}
\end{equation}
Using $X(L+\frac{1}{n}J) = I$ and $XJ=J$, we have $XL = I-\frac{1}{n}J$. Hence by $(\ref{RL})$,
\begin{equation}\label{eqn11}
\begin{aligned}
RL + 2I&=J \tilde{X}L - 2(I-\frac{1}{n}J)+2I \\
&=J \tilde{X}L+\frac{2}{n}J.
\end{aligned}
\end{equation}
By (i),
\begin{equation}\label{eqn12}
\begin{aligned} 
\1\tau' &=\1 (\1' \tilde{X} L' + \frac{2}{n} \1') \\
&= J \tilde{X}L'  + \frac{2}{n} J.
\end{aligned}
\end{equation}
From $(\ref{eqn11})$ and $(\ref{eqn12})$,  we get  $$RL+2I-J \tilde{X}L =\1 \tau'- J \tilde{X} L',$$and hence $$RL+2I=\1 \tau'+ J \tilde{X} M.$$
The proof of (iv) is done. 
Using part (i),
\[\1'\tau = \1'L\tilde{X}\1+\frac{2}{n}\1'\1 = 2.\] This proves (v). \\
 By using (i), (ii) and $M=L-L'$, we have 
\begin{equation}\label{eqn16}
\begin{aligned}
\tau'R\tau &=  \Big( \1'\tilde{X}L+\frac{2}{n}\1'-\1'\tilde{X}M \Big) R \Big( L\tilde{X}\1+\frac{2}{n}\1 \Big) \\ 
&= \1'\tilde{X}L'RL\tilde{X}\1 + \frac{2}{n} \1'\tilde{X}L'R\1 + \frac{2}{n} \1'RL\tilde{X}\1 + \frac{4}{n^2} \1'R\1. 
\end{aligned}
\end{equation}
As $L \1=0$, $L'\1 = 0$ and $R=\tilde{X} J + J\tilde{X} -2X$,
\begin{equation}\label{eqn17}
\begin{aligned}
\1'\tilde{X}L'RL\tilde{X}\1 &= \1'\tilde{X}L'(\tilde{X}J+J\tilde{X}-2X)L\tilde{X}\1 \\ &= -2\1'\tilde{X}L'XL\tilde{X}\1.
\end{aligned}
\end{equation}
As $XL=I - \frac{1}{n}J$, by $(\ref{eqn17})$,
\begin{equation}\label{eqn18}
\begin{aligned}
\1'\tilde{X}L'RL\tilde{X}\1 =  -2\1'\tilde{X}L'\Big(I-\frac{1}{n}J\Big)\tilde{X}\1 &= -2\1'\tilde{X}L'\tilde{X}\1 \\ &= -2\tilde{x}'L\tilde{x}.
\end{aligned}
\end{equation} 
By $(\ref{eqn11})$,
\begin{equation}\label{eqn19}
\begin{aligned}
\1'RL\tilde{X}\1 &= \1'\Big(J \tilde{X}L+\frac{2}{n}J-2I\Big)\tilde{X}\1 &= (n \1'\tilde{X}L)\tilde{X}\1 \\ &= n\tilde{x}'L\tilde{x}. 
\end{aligned}
\end{equation}
Since $L'\1=0$ and $X\1 = \1$,
\begin{equation} \label{eqn20}
\begin{aligned}
\1'\tilde{X}L'R\1 &= \1'\tilde{X}L'(\tilde{X}J+J\tilde{X}-2X)\1 \\ &= n\1'\tilde{X}L'\tilde{X}\1+2\1'\tilde{X}L'\1 \\ &= n\tilde{x}'L\tilde{x}.
\end{aligned}
\end{equation}
From $R=\tilde{X} J + J\tilde{X} -2X$ and $X \1=1$, we have
\begin{equation}\label{eqn21}
\1'R\1 = 2n~ \tr(X)-2n = 2n~\tr(L^\dag).
\end{equation}
Substituting $(\ref{eqn18})$, $(\ref{eqn19})$, $(\ref{eqn20})$ and $(\ref{eqn21})$ in $(\ref{eqn16})$, we get
\begin{equation*}
\begin{aligned}
\tau'R\tau =2 \tilde{x}' L \tilde{x} + \frac{8}{n} \tr(L^{\dagger}). 
\end{aligned}
\end{equation*}
Since $L+L'$ is positive semidefinite, $\tilde{x}' L \tilde{x} \geq 0$. As trace of $L^\dag$ is also positive, we get (vii). The proof is complete.
\end{proof}

\begin{thm}\label{11}
\[ R^{-1} = - \frac{1}{2} L+ \frac{1}{\tau' R \tau} (\tau (\tau'+1'\diag(L^\dag)M)),\]
where $M = L-L^T$. 
\end{thm}

\begin{proof}
By item (iii) of Lemma \ref{3},
\[LR + 2 I= \tau \1'. \]
In view of item (v) of the previous Lemma, $\1' \tau=2$. So,
\[LR \tau + 2\tau = (\1' \tau)\tau=2 \tau. \]
This implies $LR \tau=0$ and since $L \in \zl$, there exists $0 \neq \alpha \in \rr$ such that
$R \tau =\alpha \1$. As $\tau' \1=2$,   we get $\alpha = \frac{1}{2}\tau' R \tau$. Therefore, 
\begin{equation} \label{trt}
R \tau= \frac{\tau' R \tau}{2} \1.
\end{equation}
Since $M \1=0$, from item (iv) of Lemma $\ref{3}$, we deduce that
\begin{equation*}
\begin{aligned}
(\tau' + \1'\tilde{X}M)(RL+2I) &= (\tau' + \1'\tilde{X}M)(\1\tau'+J\tilde{X}M)  \\
&=2(\tau' + \1'\tilde{X}M).
\end{aligned}
\end{equation*}
After simplification the above equation leads to
\begin{equation*} \label{taue}
(\tau' + \1' \tilde{X} M)RL=0. 
\end{equation*}
We now claim that $(\tau' + \1' \tilde{X} M)R \neq 0$.
If not, then $\tau' R \tau + \1' \tilde{X} M R \tau=0$. By $(\ref{trt})$, $R \tau$ is a multiple of $\1$. So, $M R \tau=0$ and hence $\tau'R \tau=0$. This contradicts the previous Lemma. Hence, $(\tau' + \1' \tilde{X} M) R \neq 0$. 
As $L \in \zl$, it follows that $$(\tau' + \1' \tilde{X} M)R=\beta \1',$$for some $\beta \neq 0$. Since $\1' \tau=2$, $\beta= \frac{1}{2}\tau' R \tau$.  Thus,
\begin{equation} \label{q}
(\tau' + \1'\tilde{X}M)R = \frac{\tau' R \tau}{2} \1'.
\end{equation}
Now,  item (iii) of Lemma \ref{3} and (\ref{q}) imply 
\begin{equation*}
\begin{aligned}
\Big(-\frac{1}{2} L+ \frac{\tau ({\tau}'+1'\tilde{X}M)}{{\tau}'R{\tau}}\Big)R &= -\frac{1}{2} LR+ \frac{1}{{\tau}'R{\tau}} \tau ({\tau}'+1'\tilde{X}M)R  \\
&= I - \frac{1}{2} \tau \1'+  \frac{1}{{\tau}'R{\tau}} \Big( \frac{\tau' R \tau}{2} \Big) \tau \1' \\
&=I.
\end{aligned}
\end{equation*}
Since $L^\dag = X - \frac{1}{n}J$, $\1'\tilde{X}M =\1'(\diag(L^\dag)+\frac{1}{n}I)M= \1'\diag(L^\dag)M$. The proof is complete.
\end{proof}
To illustrate the inverse formula in Theorem $\ref{11}$, we consider the resistance matrix of Example $\ref{12}$.
\begin{exmm}\label{17} \rm
Consider the resistance matrix in Example \ref{12}.
\begin{equation}
R = 
\left[
{\begin{array}{rrrrrr}
0 & \frac{2}{3} & \frac{5}{3} & \frac{17}{12} & \frac{13}{12} & \frac{5}{3} \\
\frac{4}{3} & 0 & 1 & \frac{3}{4} & \frac{5}{12} & 1 \\
\frac{7}{3} & 1 & 0 & \frac{7}{4} & \frac{17}{12} & 2 \\
\frac{19}{12} & \frac{5}{4} & \frac{9}{4} & 0 & \frac{2}{3} & \frac{5}{4} \\
\frac{11}{12} & \frac{7}{12} & \frac{19}{12} & \frac{4}{3} & 0 & \frac{7}{12} \\
\frac{1}{3} & 1 & 2 & \frac{7}{4} & \frac{17}{12} & 0
\end{array}}
\right].
\end{equation}
Then we have the following:
\[
\tau = 
\left[
{\begin{array}{rrrrrr}
\frac{2}{3} & -\frac{5}{6} & 1 & \frac{2}{3} & \frac{1}{6} & \frac{1}{3}
\end{array}}
\right]',
\]
\[
{\tau}'+1'\diag{(L^\dag)}M 
= \left[
{\begin{array}{rrrrrr}
\frac{1}{3} & -\frac{3}{4} & 1 & \frac{3}{4} & \frac{1}{12} & \frac{7}{12}
\end{array}}
\right] ,\]
and
\begin{equation}\label{eqn7}
\tau'R\tau = \frac{67}{12}.
\end{equation}
We now have
\begin{small}
\begin{equation*}
\begin{aligned}
R^{-1} &= - \frac{1}{2} L+ \frac{1}{{\tau}'R{\tau}}(\tau ({\tau}'+1'\diag(L^\dag)M)) \\ 
&=\left[
{\begin{array}{rrrrrr}
-\frac{185}{402} & \frac{55}{134} & \frac{8}{67} & \frac{6}{67} & \frac{2}{201} & \frac{14}{201} \\
-\frac{10}{201} & -\frac{93}{67} & \frac{47}{134} & \frac{26}{67} & \frac{98}{201} & -\frac{35}{402} \\
\frac{4}{67} & \frac{49}{134} & -\frac{43}{134} & \frac{9}{67} & \frac{1}{67} & \frac{7}{67} \\
\frac{8}{201} & -\frac{6}{67} & \frac{8}{67} & -\frac{55}{134} & \frac{205}{402} & \frac{14}{201} \\
\frac{2}{201} & \frac{32}{67} & \frac{2}{67} & \frac{3}{134} & -\frac{401}{402} & \frac{104}{201} \\
\frac{209}{402} & -\frac{3}{67} & \frac{4}{67} & \frac{3}{67} & \frac{1}{201} & -\frac{187}{402}
\end{array}}
\right].
\end{aligned}
\end{equation*}
\end{small}
\end{exmm}

\subsection{Determinant of the resistance matrix}
By using Theorem $\ref{11}$, we compute an expression for the determinant of the resistance matrix.

\begin{corollary}\label{14}
\[ \det(R) = (-1)^{n-1} 2^{n-3} \frac{{\tau}'R \tau}{\kappa(G)}.\]
\end{corollary}
\begin{proof}
By using Theorem $\ref{11}$ and (P9), we have
\begin{equation*}
\begin{aligned}
\det(R^{-1})  &= \frac{1}{\tau'R\tau}(\tau'+\1'\diag(L^\dag)M)\adj(-\frac{1}{2}L)\tau \\ &= \bigg(-\frac{1}{2}\bigg)^{n-1} \frac{\kappa(G)}{\tau'R\tau} (\tau'+\1'\diag(L^\dag)M)J\tau \\ &= \bigg(-\frac{1}{2}\bigg)^{n-1}\frac{\kappa(G)}{\tau'R\tau} \tau'J\tau.
\end{aligned}
\end{equation*}
Since $\1'\tau = 2$, it follows that
\[ \det(R) = (-1)^{n-1} 2^{n-3} \frac{{\tau}'R \tau}{\kappa(G)}.\]
\end{proof}

\begin{exmm} \rm
Consider the directed graph $G$ on six vertices given in Figure \ref{digrapheg}(a). $G$ has two oriented spanning trees $T_1$ and $T_2$ (see Figure \ref{deteg}(a) and \ref{deteg}(b), respectively) rooted at vertex $1$.
\begin{figure}[!h]
~~~~~~~\subfloat[]{\begin{tikzpicture}[shorten >=1pt, auto, node distance=3cm, ultra thick,
   node_style/.style={circle,draw=black,fill=white !20!,font=\sffamily\Large\bfseries},
   edge_style/.style={draw=black, ultra thick}]
\node[vertex] (1) at  (0,0) {$1$};
\node[vertex] (2) at  (2,0) {$2$};
\node[vertex] (3) at  (3.5,-1) {$3$};
\node[vertex] (4) at  (2,-2) {$4$};
\node[vertex] (5) at  (0,-2) {$5$};
\node[vertex] (6) at  (-1.5,-1) {$6$};
\draw[edge]  (3.2,-0.7) to (2.39,-0.13);
\draw[edge]  (2) to (4);
\draw[edge]  (4) to (5);
\draw[edge]  (5) to (6);
\draw[edge]  (6) to (1);
\end{tikzpicture}}
~~~~~~~~~~~~~~ 
\subfloat[]{\begin{tikzpicture}[shorten >=1pt, auto, node distance=3cm, ultra thick,
   node_style/.style={circle,draw=black,fill=white !20!,font=\sffamily\Large\bfseries},
   edge_style/.style={draw=black, ultra thick}]
\node[vertex] (1) at  (0,0) {$1$};
\node[vertex] (2) at  (2,0) {$2$};
\node[vertex] (3) at  (3.5,-1) {$3$};
\node[vertex] (4) at  (2,-2) {$4$};
\node[vertex] (5) at  (0,-2) {$5$};
\node[vertex] (6) at  (-1.5,-1) {$6$};
\draw[edge]  (3.2,-0.7) to (2.39,-0.13);
\draw[edge]  (1.8,-0.35) to (0.385,-1.8);
\draw[edge]  (4) to (5);
\draw[edge]  (5) to (6);
\draw[edge]  (6) to (1);
\end{tikzpicture}}
\caption{(a) spanning tree $T_1$ (b) spanning tree $T_2$} \label{deteg}
\end{figure}
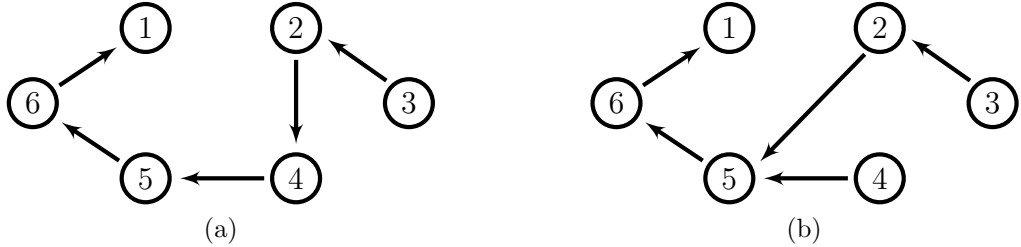
Thus, $\kappa(G) = 2$. From Example \ref{17}, $\tau'R\tau = \frac{67}{12}$. By Corollary \ref{14}, we have
\begin{equation}
\begin{aligned}
\det(R) &= (-1)^{n-1} 2^{n-3} \frac{{\tau}'R \tau}{\kappa(G)} \\ &= -\frac{67}{3}.
\end{aligned}
\end{equation}
\end{exmm}

\section{Cofactor sums of the resistance matrix}
Let $\Omega_1,\Omega_2 \subset \{1,2,\ldots,n\}$ be non-empty and $ |\Omega_1| = |\Omega_2|$. Define $\eta:= |\Omega_1| = |\Omega_2|$. We now derive an identity for computing the sum of all the entries in the cofactor matrix of $R[\Omega_1,\Omega_2]$.
We shall use the following elementary lemma repeatedly. The proof is immediate. 
\begin{lem}\label{1.4} 
Let B be an $m \times m$ matrix, $\beta \in \rr$ and 
 \[
   A=
  \left[ {\begin{array}{cc}
   B & \frac{1}{\displaystyle \beta}\1_m \\
   \frac{1}{\displaystyle \beta} \1'_m & 0 \\
  \end{array} } \right]
.\] Then, 
\[ \csum(B) = -{\beta}^2\det(A).\]
\end{lem}
We now obtain the following identity.
\begin{lem}\label{1.5}
Let $S$ be a $n \times n$ matrix. Suppose rank$(S)= n-1$, $S \1=0$ and $S' \1=0$. Then,
\[\csum(S[\Omega_1,\Omega_2]) = (-1)^{\alpha(\Omega_1)+\alpha(\Omega_2)} n^2 \mathit{\gamma} \det(S^ \dag [\Omega_2^c,\Omega_1^c]).\]
where $\gamma$ is the common cofactor value of $S$.
\end{lem}
\begin{proof}
Let \[
   A:=
  \left[ {\begin{array}{cc}
   S & \frac{1}{\sqrt{n}} \1 \\
   \frac{1}{\sqrt{n}} \1' & 0 \\
  \end{array} } \right].
\]
Then $A$ is non-singular and  in fact, 
\begin{equation} \label{nsingA}
   A^{-1}=
  \left[ {\begin{array}{cc}
   S^{\dag} & \frac{1}{\sqrt{n}} \1 \\
   \frac{1}{\sqrt{n}}\1' & 0 \\
  \end{array} } \right].
\end{equation}
Define $\widetilde{S}:=S[\Omega_1,\Omega_2]$.
By Lemma ${\ref{1.4}}$,  
\begin{equation} \label{csuml}
\csum(\widetilde{S})= -n \det\Bigg(\left[ {\begin{array}{cc}
   \widetilde{S} & \frac{1}{\sqrt{n}}\1_{\eta }\\
   \frac{1}{\sqrt{n}}  \1'_{\eta} & 0 \\
  \end{array} } \right]\Bigg).
  \end{equation}
Define $$\Delta_1:=\Omega_{1} \cup \{n+1\}  ~~~\mbox{and}~~~\Delta_2:=\Omega_2 \cup \{n+1\}.$$
Then, 
 \[ A[\Delta_1,\Delta_2]=\left[ {\begin{array}{cc}
   \widetilde{S} & \frac{1}{\sqrt{n}}\1_{\eta }\\
   \frac{1}{\sqrt{n}}  \1'_{\eta} & 0 \\
  \end{array} } \right] .\]
  By rewriting equation $(\ref{csuml})$, we have
\begin{equation}\label{eqn1}
\csum(\widetilde{S})= -n \det(A[\Delta_1,\Delta_2]).
\end{equation}
By Jacobi's formula (P6)
\begin{equation}\label{eqn2}
\det(A[\Delta_1,\Delta_2]) = (-1)^{\alpha(\Omega_1)+\alpha(\Omega_2)} \frac{\det(A^{-1}[\Delta_2^{c}, \Delta_1^c])}{\det(A^{-1})}.
\end{equation}
From $(\ref{eqn1})$ and $(\ref{eqn2})$, we get
\begin{equation}\label{eqn23}
\csum(\widetilde{S})= (-1)^{\alpha(\Omega_1)+\alpha(\Omega_2)+1} n \frac{\det(A^{-1}[\Delta_2^c, \Delta_1^{c}])}{\det(A^{-1})}.
\end{equation}
 Using equation $(\ref{nsingA})$,
 \begin{equation}\label{eqn40}
  A^{-1}[\Delta_{2}^{c},\Delta_{1}^{c}]= S^{\dag}[\Omega_2^c ,\Omega_1^c].
 \end{equation}
Again applying Lemma ${\ref{1.4}}$, $$\det(A) = -\frac{1}{n}\csum(S) = -n\gamma.$$ where $\gamma$ is the common cofactor value of $S$. So, 
\begin{equation}\label{eqn41}
\det(A^{-1})= -\frac{1}{n\gamma}.
\end{equation}
By  $(\ref{eqn23})$,$(\ref{eqn40})$ and $(\ref{eqn41})$, 
\[ \csum(\widetilde{S})= (-1)^{\alpha(\Omega_1)+\alpha(\Omega_2)} n^2 \mathit{\gamma} \det(S^\dag [\Omega_2^c ,\Omega_1^c]). \]
The proof is complete.
\end{proof}

\begin{lem}\label{1.7}
Let A be a ${n}\times{n}$ matrix and let $P = I - \frac{1}{n}\1\1'$.  
Define $S:=PAP$. Then,
 \[\csum(A) = \csum(S) ~~\mbox{and} ~~\csum (A[\Omega_1,\Omega_2])=\csum (S[\Omega_1,\Omega_2]). \]
\end{lem}
\begin{proof}
We begin by noting that
\begin{equation} \label{cofsumd}
\begin{aligned}
S &= (I - \frac{1}{n}\1\1')A(I - \frac{1}{n}\1\1') \\ &= A - \frac{1}{n}A\1 \1'- \frac{1}{n}\1 \1'A+  \frac{\1'A\1}{n^2}(\1 \1').
\end{aligned}
\end{equation}
Let $e:=\1_\eta$.
By $(\ref{cofsumd})$, 
\begin{equation} \label{cofsumd1}
S[\Omega_1,\Omega_2]=A[\Omega_1,\Omega_2]+ue'+ev'+ \beta ee', 
\end{equation}
for some vectors $u$, $v$ in $\rr^\eta$ and for some real scalar $\beta$. 
We now claim that if $x \in \rr^\eta$, and if $B$ is an
$\eta \times \eta$ matrix, then $$\csum(B + x\1_{\eta}') = \csum(B).$$ Using (P9), we get
\begin{equation}\label{eqn4}
\begin{aligned}
  \csum(B + x\1_\eta') &= \1'\adj(B + x\1_\eta')\1 \\ &= \det(B + x\1_\eta'+\1_\eta\1_\eta')-\det(B + x\1_\eta') \\ &= \det(B + (x+\1_\eta)\1_\eta')-\det(B + x\1_\eta') \\ &= \det(B)+\1_\eta'\adj(B)(x+\1_\eta)-\det(B)-\1_\eta'\adj(B)x \\ &= \1_\eta'\adj(B)\1_\eta= \csum(B). 
\end{aligned}
 \end{equation}
Similarly, we see that
\begin{equation}\label{eqn24}
\begin{aligned}
\csum(B + \1_\eta x') = \csum(B).
\end{aligned}
\end{equation}  
Repeatedly using $(\ref{eqn4})$ and $(\ref{eqn24})$ in $(\ref{cofsumd})$ and $(\ref{cofsumd1})$, we obtain
\[\csum(A) = \csum(S) ~~\mbox{and} ~~\csum (A[\Omega_1,\Omega_2])=\csum (S[\Omega_1,\Omega_2]). \]
This completes the proof.
\end{proof}
By Lemma \ref{1.5} and \ref{1.7}, we now obtain the following result. 
\begin{thm}\label{1.6}
Let $S$ be an $n \times n$ matrix such that $\rank(S)=n-1$, $S\1=0$ and $S'\1=0$. Define $D = [d_{ij}]$ by
\[D= \diag(S)J+J\diag(S)-2S.\] 
Then, 
\[\csum(D[\Omega_1,\Omega_2]) = (-1)^{\alpha(\Omega_1)+\alpha(\Omega_2)+\eta-1} 2^{\eta-1} n^2 \mathit{\gamma} \\ \det(S^\dag [\Omega_2^c,\Omega_1^c]),\]
where $\gamma$ is the common cofactor value of $S$.
\end{thm}

\begin{proof}
Pre and post multiplying by $P$ in the equation 
\[ D=\diag(S)J+J \diag(S) -2 S,\] we have
\[ PDP=-2S.\]
Thus, by Lemma $\ref{1.7}$, 
\begin{equation}\label{eqn25}
\begin{aligned}
\csum(S[\Omega_1,\Omega_2]) &= \csum ( -\frac{1}{2}D[\Omega_1,\Omega_2]) \\ &=  \Big(\frac{-1}{~~2}\Big)^{\eta-1} \csum(D[\Omega_1,\Omega_2]).
\end{aligned}
\end{equation}
Using Lemma ${\ref{1.5}}$ in $(\ref{eqn25})$, we get
\begin{equation*}
\csum(D[\Omega_1,\Omega_2]) = (-1)^{\alpha(\Omega_1)+\alpha(\Omega_2)+\eta-1} 2^{\eta-1} n^2 \mathit{\gamma} \det(S^\dag [\Omega_2^c,\Omega_1^c]).
\end{equation*}
The proof is complete. 
\end{proof}

It can be noted that Theorem \ref{csumR}  follows from Theorem \ref{1.6} immediately.  Applying Theorem \ref{csumR}  to resistance matrices of strongly connected balanced directed graphs, we get the following.
\begin{thm}\label{rcofsum}
Let $G$ be a strongly connected balanced directed graph with vertex set $\{1,2,\ldots,n\}$, Laplacian matrix $L$ and resistance matrix $R$. Then the following items hold.
\begin{enumerate}
    \item[(i)] $\csum(R[\Omega_1,\Omega_2]) = (-1)^{\alpha(\Omega_1)+\alpha(\Omega_2)+\eta-1}  \frac{2^{\eta-1}}{\kappa(G)}\det(L [\Omega_2^c,\Omega_1^c]).$
\item[(ii)] For every distinct $i,j \in \{1,2,..,n\}$,
\[r_{ij}+r_{ji}  = \frac{2}{\kappa(G)}\det(L[\{i,j\}^c,\{i,j\}^c]).\]
\end{enumerate}
\end{thm}
\begin{proof}
\begin{enumerate}
    \item[(i)] Since 
\[R = \diag(L^\dag)J+J\diag(L^\dag)-2L^\dag, \] by Theorem $\ref{1.6}$ it follows that
\begin{equation}\label{eqn36}
\csum(R[\Omega_1,\Omega_2]) = (-1)^{\alpha(\Omega_1)+\alpha(\Omega_2)+\eta-1} 2^{\eta-1} n^2 \mathit{\delta} \\ \det(L [\Omega_2^c,\Omega_1^c]),
\end{equation} 
where $\mathit{\delta}$ is the common cofactor value of $L^\dag$. Let \[
   A:=
  \left[ {\begin{array}{cc}
   L & \frac{1}{\sqrt{n}} \1 \\
   \frac{1}{\sqrt{n}} \1' & 0 \\
  \end{array} } \right].
\]
Then $A$ is non-singular and,
\begin{equation*} \label{nsingA2}
   A^{-1}=
  \left[ {\begin{array}{cc}
   L^\dag & \frac{1}{\sqrt{n}} \1 \\
   \frac{1}{\sqrt{n}}\1' & 0 \\
  \end{array} } \right].
\end{equation*}
By Lemma $\ref{1.4}$, we have $\csum(L^\dag) = -n \det(A^{-1})$ and $\csum(L) = -n \det(A)$. Thus, 
\begin{equation*}\label{eqn32}
\csum(L) = \frac{n^2}{\csum(L^\dag)} = \frac{1}{\delta}.
\end{equation*}
and hence
\begin{equation}\label{eqn31}
\kappa(G) = \frac{\csum(L)}{n^2} = \frac{1}{n^2 \delta}.
\end{equation}
By $(\ref{eqn36})$ and $(\ref{eqn31})$, we have
\[\csum(R[\Omega_1,\Omega_2]) = (-1)^{\alpha(\Omega_1)+\alpha(\Omega_2)+\eta-1}  \frac{2^{\eta-1}}{\kappa(G)}\det(L [\Omega_2^c,\Omega_1^c]).\]
The proof of (i) is complete.
\item[(ii)] Let $i,j \in \{1,2,..,n\}$ be such that $i \neq j$. Substituting $\Omega_1 = \Omega_2 = \{i,j\}$ in (i), we get
\begin{equation}\label{eqn28}
\csum(R[\Omega_1,\Omega_2]) = (-1)^{2i+2j+1}  \frac{2}{\kappa(G)}\det(L [\Omega_2^c,\Omega_1^c]).
\end{equation}
As $\csum (R[\Omega_1,\Omega_2])=-(r_{ij}+r_{ji})$, by $(\ref{eqn28})$, 
\[r_{ij}+r_{ji} = \frac{2}{\kappa(G)}\det(L[\{i,j\}^c,\{i,j\}^c]).\]
This completes the proof of (ii).
\end{enumerate}
\end{proof}
To illustrate the above theorem, we present the following example.
\begin{exmm} \rm
Consider the directed graph $G$ on four vertices given in Figure \ref{csumeg}(a).
\begin{figure}[!h]
~~~~~\subfloat[]{\begin{tikzpicture}[shorten >=1pt, auto, node distance=3cm, ultra thick,
   node_style/.style={circle,draw=black,fill=white !20!,font=\sffamily\Large\bfseries},
   edge_style/.style={draw=black, ultra thick}]
\node[vertex] (1) at  (0,0) {$1$};
\node[vertex] (2) at  (2,0) {$2$};
\node[vertex] (3) at  (2,-2) {$3$};
\node[vertex] (4) at  (0,-2) {$4$};
\draw[edge]  (1) to (2);
\draw[edge]  (2) to (3);
\draw[edge]  (3) to (4);
\draw[edge]  (4) to (1);
\draw[edge]  (1) to (3);
\draw[edge]  (1.6,-1.8) to (0.2,-0.4);
\end{tikzpicture}}
~~~~~~~~~~~~~~
\subfloat[]{\begin{tikzpicture}[shorten >=1pt, auto, node distance=3cm, ultra thick,
   node_style/.style={circle,draw=black,fill=white !20!,font=\sffamily\Large\bfseries},
   edge_style/.style={draw=black, ultra thick}]
\node[vertex] (1) at  (0,0) {$1$};
\node[vertex] (2) at  (2,0) {$2$};
\node[vertex] (3) at  (2,-2) {$3$};
\node[vertex] (4) at  (0,-2) {$4$};
\draw[edge]  (1) to (2);
\draw[edge]  (2) to (3);
\draw[edge]  (3) to (4);
\end{tikzpicture}}
~~~~~~~~~~~~~~
\subfloat[]{\begin{tikzpicture}[shorten >=1pt, auto, node distance=3cm, ultra thick,
   node_style/.style={circle,draw=black,fill=white !20!,font=\sffamily\Large\bfseries},
   edge_style/.style={draw=black, ultra thick}]
\node[vertex] (1) at  (0,0) {$1$};
\node[vertex] (2) at  (2,0) {$2$};
\node[vertex] (3) at  (2,-2) {$3$};
\node[vertex] (4) at  (0,-2) {$4$};
\draw[edge]  (2) to (3);
\draw[edge]  (3) to (4);
\draw[edge]  (1) to (3);
\end{tikzpicture}}
\caption{(a) Graph $G$, (b) spanning tree $T_1$ and (c) spanning tree $T_2$} \label{csumeg}
\end{figure}
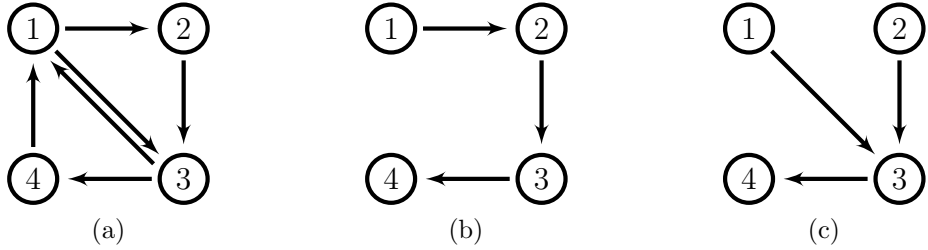
$G$ has two oriented spanning trees $T_1$ and $T_2$ rooted at vertex $4$ (see Figure \ref{csumeg}(b) and \ref{csumeg}(c)). Thus, $\kappa(G) = 2$.
The Laplacian and resistance matrices of $G$ are
\begin{equation*}
L = \left[
{\begin{array}{rrrrrr}
2 & -1 & -1 & 0 \\
0 & 1 & -1 & 0 \\
-1 & 0 & 2 & -1 \\
-1 & 0 & 0 & 1
\end{array}}
\right]
~ \text{and}~
R = 
\left[
{\begin{array}{rrrrrr}
0 & \frac{3}{4} & \frac{1}{2} & \frac{5}{4} \\
\frac{5}{4} & 0 & \frac{3}{4} & \frac{3}{2} \\
\frac{1}{2} & \frac{5}{4} & 0 & \frac{3}{4} \\
\frac{3}{4} & \frac{3}{2} & \frac{5}{4} & 0
\end{array}}
\right].
\end{equation*} Let $\Omega_1=\{1,2\}$  and
$\Omega_{2}=\{1,4\}$. Now,
\begin{equation*}
R[\Omega_1,\Omega_2] = \left[
{\begin{array}{rrrrrr}
0 & \frac{5}{4} \\
\frac{5}{4} & \frac{3}{2}
\end{array}}
\right], ~~ \csum(R[\Omega_1,\Omega_2]) = -1, 
\end{equation*}
\begin{equation*}
L[\Omega_2^c, \Omega_1^c] = \left[
{\begin{array}{rrrrrr}
-1 & 0 \\
2 & -1
\end{array}}
\right],~~
\det(L[\Omega_2^c, \Omega_1^c] ) = 1 ~\text{and}~ \alpha(\Omega_1)+\alpha(\Omega_2)+\eta-1 = 9.
\end{equation*}
Hence,
\begin{equation*}
 (-1)^{\alpha(\Omega_1)+\alpha(\Omega_2)+\eta-1} \frac{2^{\eta-1}}{\kappa(G)}\det(L [\Omega_2^c,\Omega_1^c]) = -1 = \csum(R[\Omega_1,\Omega_2]).
\end{equation*}
\end{exmm}

\section*{Acknowledgements}
The first author is supported by Department of science and Technology -India under the project MATRICS (MTR/2017/000342).

\end{document}